\documentclass[11pt,a4paper]{amsart}
\allowdisplaybreaks 

\usepackage{color}
\usepackage{amsmath}
\usepackage{amssymb}
\usepackage{amsthm}
\usepackage{amsfonts}
\usepackage{latexsym}
\usepackage{hyperref}
\usepackage{tikz}
\usetikzlibrary{matrix,arrows}

\newtheorem{theorem}{Theorem}[section]
\newtheorem{corollary}[theorem]{Corollary}
\newtheorem{lemma}[theorem]{Lemma}

\numberwithin{equation}{section}

\def \bR {\mathbb R}

\def \bR {\mathbb R}
\def \bR {\mathbb R}

\def \cE {\mathcal E}
\def \cF {\mathcal F}

\def \cM {\mathcal M}

\def \cR {\mathcal R}
\def \cS {\mathcal S}
\def \cR {\mathcal R}
\def \cR {\mathcal R}

\def \la {\lambda}
\def \ph {\varphi}

\def \eps {\varepsilon}
\def \lan {\langle}
\def \ran {\rangle}

\def \half{\frac12}
\def \inv{^{-1}}

\def\be{\begin{equation}}
\def\ee{\end{equation}}
\def\bes{\begin{equation*}}
\def\ees{\end{equation*}}
\def\bea{\begin{equation}\begin{aligned}}
\def\eea{\end{aligned}\end{equation}}
\def\beas{\begin{equation*}\begin{aligned}}
\def\eeas{\end{aligned}\end{equation*}}

\title[A maximal  restriction theorem]{A maximal restriction theorem\\ and Lebesgue points of functions in $\cF(L^p)$}

\author{Detlef M\"uller}
\address{Christian-Albrechts-Universit\"at zu Kiel, Mathematisches Seminar,
Ludewig-Meyn-Str. 4, D-24118 Kiel, Germany } 
\email{{\tt mueller@math.uni-kiel.de}}

\author{Fulvio Ricci}
\address{Scuola Normale Superiore, Piazza dei Cavalieri
7, 56126 Pisa, Italy } 
\email{{\tt fricci@sns.it}}

\author{James Wright}
\address{School of Mathematics and Maxwell Institute for Mathematical Sciences University of Edinburgh, Edinburgh EH9 3FD, Scotland} 
\email{{\tt J.R.Wright@ed.ac.uk}}

\thanks{The first author has been supported by the DFG grant DFG-Project MU 761/11-1}

\begin{document}

\maketitle

\section{Introduction}

The  restriction problem for the Fourier transform in $\bR^n$ was introduced by E. M. Stein, who proved the first result in any dimension \cite[p. 28]{Fe}, later improved by the sharper Stein-Tomas method [T]. Since then more and more sophisticated techniques have been introduced to attack the still open problems in this area, concerning the maximal range of exponents for which the restriction inequality holds.

In  two-dimensions, the restriction estimate  for the circle had been proved already, in an almost optimal range of exponents, by Fefferman and Stein \cite[p. 33]{Fe}. Shortly later, sharp estimates were obtained by Zygmund \cite{Z} for the circle and by Carleson and Sj\"olin \cite{CS} and Sj\"olin \cite{Sj} for a  class of curves including strictly convex $C^2$ curves.
  
The present paper does not mean to proceed along these lines, but rather to propose a reflection on the measure-theoretic meaning of the restriction phenomenon and possibly suggest some related problems.
\smallskip

A restriction theorem is usually meant as a family of {\it apriori} inequalities
\be\label{restriction_inequality}
\big\|\widehat f_{|_S}\big\|_{L^q(S,\mu)}\le C\big\|f\|_{L^p(\bR^n)}\ ,
\ee
where $f\in \cS(\bR^n)$, $S$ is a surface with appropriate curvature properties, and $\mu$  a suitably weighted finite surface measure on $S$.  The validity of such an inequality implies the existence of a bounded {\it restriction operator} $\cR:L^p(\bR^n)\longrightarrow L^q(S,\mu)$ such that $\cR f=\widehat f_{|_S}$ when $f$ is a Schwartz function.

\smallskip

In general terms our question is: assuming that \eqref{restriction_inequality} holds, what is the ``intrinsic'' pointwise relation between $\cR f$ and $\widehat f$ for a general $L^p$-function $f$? 

A partial answer follows directly from the restriction inequality.  Assume that \eqref{restriction_inequality} holds for given $p,q$. This forces the condition $p<2$, so that $\widehat f\in L^{p'}$. Fix an approximate identity $\chi_\eps(x)=\eps^{-n}\chi(x/\eps)$ with $\chi\in\cS(\bR^n)$, $\int\chi=1$. Then, with $\psi=\cF\inv\chi$,
$$
\widehat f*\chi_\eps=\widehat{f\psi(\eps\cdot)}
$$
is well defined on $S$ and coincides with $\cR\big(f\psi(\eps\cdot)\big)$. Moreover, $f\psi(\eps\cdot)\to f$ in $L^p(\bR^n)$, so that $(\widehat f*\chi_\eps)_{|_S}\to \cR f$ in $L^q(S,\mu)$. Hence, for a subsequence $\eps_k\to0$, the $\chi_{\eps_k}$-averages of $\widehat f$ converge pointwise to $\cR f$ $\mu$-a.e.


It is natural to ask if the limit over all $\eps$ exists $\mu$-a.e.  We give  positive answers in two dimensions to this and related questions.

We recall that, for a curve $S$ in the plane, necessary conditions on $p,q$ for having \eqref{restriction_inequality} are $p<\frac43$ and $p'\ge 3q$ and that they are also sufficient when $S$ is  $C^2$  with nonvanishing curvature and $\mu$ is the arclength measure, or, more generally, when $S$ is just $C^2$ and convex, and $\mu$ is the {\it affine arclength measure}~\cite{Sj}. Notice that the two measures differ by a factor comparable to the $\frac13$ power of the curvature, so that the affine arclength is concentrated on the set of points with nonvanishing curvature and ordinary arclength is damped near these points.

\begin{theorem}\label{lebesgue}
Let $S$ be a $C^2$ curve in $\bR^2$ and $f\in L^p(\bR^2)$.
\begin{enumerate}
\item[\rm(i)] Assume that $1\le p<\frac43$ and let $\chi\in\cS(\bR^2)$ with $\int\chi=1$. Then, with respect to arclength measure, for almost every $x\in S$ at which the curvature does not vanish, $\lim_{\eps\to0}\widehat f*\chi_\eps(x)=\cR f(x)$.
\item[\rm(ii)] Assume that $1\le p<\frac87$. Then, with respect to arclength measure, almost every $x\in S$ at which the curvature does not vanish is a Lebesgue point for $\widehat f$ and the regularized value of $\widehat f$ at $x$ coincides with $\cR f(x)$.
\end{enumerate}
\end{theorem}

Several questions remain open, regarding extensions to less regular curves, to other values of $p$ in the  range $\frac87\le p<\frac43$, or to higher dimensions. We just mention here that, in dimension $d\ge3$, our method gives results for a class of  curves including $\Gamma(t)=(t,t^2,\dots,t^d)$.

Theorem \ref{lebesgue} is a direct consequence of certain ``maximal restriction theorems'' concerning  restrictions to $S$ of  truncated maximal functions of the Fourier transform. Since  maximal restriction inequalities may also have an intrinsic interest, we go beyond what is strictly needed to deduce Theorem \ref{lebesgue} and consider (truncated) two-parameter  maximal functions, such as the strong maximal function, relative to any coordinate system in $\bR^2$.

In Theorem \ref{vertical} we prove that, for a convex $C^2$ curve, the two-parameter maximal operator defined  in \eqref{Mv}, is $L^p-L^q$ bounded for $p,q$ in the full range of validity of the restriction theorem, with the $L^q$-norm on $S$ relative to  affine arc-length measure. 

In Corollary \ref{tildeM} we deduce   the same $L^p-L^q$ estimates, but in the smaller range $p<\frac87$, for the truncated strong maximal function, which does not only control averages of $\hat f$, but also those  of~$|\hat f|$. 

The proof is based on the Kolmogorov-Seliverstov-Plessner   linearization method \cite[Ch. XIII]{Zb}. This leads to proving uniform estimates for a family of linear operators to which a modification of the basic approach of [CS,Z] for curves in $\bR^2$ can be applied.  For this reason our method is limited to the two-dimensional context.
Unfortunately, the usual $TT^*$ method of Stein-Tomas does not seem to be applicable,  even for the Hardy-Littlewood maximal function.

\section{The strong maximal function of $\widehat f$ along a curve
}\label{sec-vertical}

Let $S=\{\Gamma(t):t\in I\}$, where $\Gamma$ is a $C^2$ curve in $\bR^2$ with nonnegative signed curvature, i.e., with $\kappa(t)=\det(\Gamma',\Gamma'')(t)\ge0$. Denote by $d\mu(t)=\kappa^{\frac13}(t)\,dt$ the pull-back to~$I$ of the affine arclength measure on~$S$.

We assume for simplicity that  $\Gamma(x)=\big(x,\ph(x)\big)$ is the graph of a convex $C^2$ function $\ph$ on a bounded  interval $I$.   Notice that the measure $\mu$ is concentrated on the set where $\kappa=\ph''>0$.

We consider  the two-parameter maximal function\footnote{Theorem \ref{vertical} also holds if $\chi\otimes\chi$ is replaced by a general $\chi\in\cS(\bR^2)$, because this can be expanded into a rapidly decreasing  series $\sum_j\chi'_j\otimes\chi''_j$.}
\be\label{Mv}
\cM f(x)=\sup_{0<\eps',\eps''<1}\Big|\int \widehat f\big(x+s,\ph(x)+t\big) \chi_{\eps'}(s) \chi_{\eps''}(t)\,ds\,dt\Big|\ ,
\ee
where $\chi_\eps(\cdot)=\eps\inv\chi(\cdot/\eps)$, with $\chi\in\cS(\bR)$, even,  with  $\int\chi=1$.

\begin{theorem}\label{vertical}
The inequality
\be
\|\cM f\|_{L^q(I,\mu)}\le C_p\|f\|_{L^p(\bR^2)}\ ,
\ee
holds for $1\le p<\frac43$ and $p'\ge 3q$.
\end{theorem}

\begin{proof} We may and shall assume  $f\in\cS(\bR^2)$ and, since $\mu$ is finite, $p'= 3q$ by H\"older's inequality.
We linearize $\cM $ by defining, for fixed measurable functions $\eps'(x),\eps''(x)$ on $I$ with values in $(0,1)$,
\bea\label{Reps}
\cR_{\eps',\eps''} f(x)&=\int \widehat f\big(x+s,\ph(x)+t\big) \chi_{\eps'}(s) \chi_{\eps''}(t)\,ds\,dt\\
&=\int f(\xi,\eta)\int e^{-i(\xi (x+s)+\eta(\ph(x)+t))}\chi_{\eps'}(s) \chi_{\eps''}(t)\,ds\,dt\,d\xi\,d\eta\\
&=\int \widehat\chi\big(\eps'(x)\xi\big)\widehat\chi\big(\eps''(x)\eta\big) e^{-i(\xi x+\eta\ph(x))}f(\xi,\eta)\,d\xi\,d\eta\ .
\eea

The formal adjoint of $\cR_{\eps',\eps''}$ is 
\bea\label{Eeps}
\cE_{\eps',\eps''} g(\xi,\eta)&=\cR^*_{\eps',\eps''} g(\xi,\eta)\\
&=\int_I\widehat\chi\big(\eps'(x)\xi\big)\widehat\chi\big(\eps''(x)\eta\big)e^{i(\xi x+\eta\ph(x))}g(x)\kappa^\frac13(x)\,dx\ .
\eea

It suffices to prove the inequality
\be\label{Eeps-estimate}
\|\cE_{\eps',\eps''} g\|_{L^{p'}(\bR^2)}\le C_p \|g\|_{L^{q'}(I,\mu)}\ ,\qquad g\in C^\infty_c(I)\ ,
\ee
uniformly in the functions $\eps'(x),\eps''(x)$. We introduce a truncation in $\xi$ and $\eta$, in order to gain
 decay at infinity for $\cE_{\eps',\eps''} g$. Fixing another function $\chi_0$ smooth on $\bR$, supported in $[-2,2]$ and equal to 1 on $[-1,1]$, we define, for  $\la\gg1$, 
\be\label{Eepslambda}
\cE_{\eps',\eps''}^\la g(\xi,\eta)=\chi_0\Big(\frac\xi\la\Big)\chi_0\Big(\frac\eta\la\Big)\int_I\widehat\chi\big(\eps'(x)\xi\big)\widehat\chi\big(\eps''(x)\eta\big)e^{i(\xi x+\eta\ph(x))}g(x)\kappa^\frac13(x)\,dx\ .
\ee

It will then suffice to prove \eqref{Eeps-estimate} with $\cE_{\eps',\eps''}$ replaced by $\cE^\la_{\eps',\eps''}$, uniformly in $\eps'(x),\eps''(x)$ and $\la$. 

We start from the identity
\be\label{square}
\|\cE_{\eps',\eps''}^\la g\|_{p'}=\big\|(\cE_{\eps',\eps''}^\la g)^2\|^\half_{p'/2}\ .
\ee

If $U$ is the open subset of $I$ where $\kappa(x)>0$, the measure $\mu$ is concentrated on $U$, so we have
\beas
(\cE_{\eps',\eps''}^\la g)^2(\xi,\eta)&=\chi_0^2\Big(\frac\xi\la\Big)\chi_0^2\Big(\frac\eta\la\Big)\int_{U^2}\widehat\chi\big(\eps'(x)\xi\big)\widehat\chi\big(\eps''(x)\eta\big)\widehat\chi\big(\eps'(y)\xi\big)\widehat\chi\big(\eps''(y)\eta\big)\\
&\qquad\qquad\qquad\qquad\qquad\qquad\qquad e^{i(\xi (x+y)+\eta(\ph(x)+\ph(y))}g(x)\kappa^\frac13(x)g(y)\kappa^\frac13(y)\,dx\,dy\\
&=\chi_0^2\Big(\frac\xi\la\Big)\chi_0^2\Big(\frac\eta\la\Big)\int_{U^2}\widehat\chi\big(\eps'(x)\xi\big)\widehat\chi\big(\eps''(x)\eta\big)\widehat\chi\big(\eps'(y)\xi\big)\widehat\chi\big(\eps''(y)\eta\big)\\
&\qquad\qquad\qquad\qquad\qquad\qquad\qquad e^{i(\xi (x+y)+\eta(\ph(x)+\ph(y))}G_0(x,y)\,dx\,dy\ ,
\eeas
with $G_0=(g\kappa^\frac13)\otimes(g\kappa^\frac13)$. 

We want to make the change of variables $z_1=x+y$, $z_2=\ph(x)+\ph(y)$.
It follows from the convexity of $\ph$ that the map $\Phi(x,y)=\big(x+y,\ph(x)+\ph(y)\big)$ is injective on each of the subsets $U^2_\pm=\{(x,y)\in U^2: x\lessgtr y\}$ and that $\det\Phi'(x,y)=\ph'(y)-\ph'(x)\ne0$ on $U^2$.

With $A=\Phi(U_+)=\Phi(U_-)$, we set, for $z=(z_1,z_2)\in A$,
\bea\label{change}
&\big(x_\pm(z),y_\pm(z)\big)=(\Phi_{|_{U^2_\pm}})\inv(z)\\
&{\eps'_1}^\pm(z)=\eps'\big(x_\pm(z)\big)\ ,\qquad {\eps'_2}^\pm(z)=\eps'\big(y_\pm(z)\big)\\
&{\eps''_1}^\pm(z)=\eps''\big(x_\pm(z)\big)\ ,\qquad {\eps''_2}^\pm(z)=\eps''\big(y_\pm(z)\big)\\
&G_\pm(z)=\frac{G_0\big(x_\pm(z),y_\pm(z)\big)}{\big|\ph'\big(x_\pm(z)\big)-\ph'\big(y_\pm(z)\big)\big|}\ .
\eea

Then 
\bea\label{E2}
\cE_{\eps',\eps''}^\la g(\xi,\eta)^2&=\chi_0^2\Big(\frac\xi\la\Big)\chi_0^2\Big(\frac\eta\la\Big)\\
&\qquad \sum_\pm\int_A\widehat\chi\big({\eps'_1}^\pm(z)\xi\big)\widehat\chi\big({\eps''_1}^\pm(z)\eta\big)\widehat\chi\big({\eps'_2}^\pm(z)\xi\big)\widehat\chi\big({\eps''_2}^\pm(z)\eta\big)e^{i(\xi z_1+\eta z_2)}G_\pm(z)\,dz\ .
\eea

We are so led to consider the operator
$$
T^\la_{\overline\eps}G(\xi,\eta)=\chi_0^2\Big(\frac\xi\la\Big)\chi_0^2\Big(\frac\eta\la\Big)\int_A\widehat\chi\big(\eps'_1(z)\xi\big)\widehat\chi\big(\eps''_1(z)\eta\big)\widehat\chi\big(\eps'_2(z)\xi\big)\widehat\chi\big(\eps''_2(z)\eta\big)e^{i(\xi z_1+\eta z_2)}G(z)\,dz\ ,
$$
for arbitrary measurable functions $\overline\eps=(\eps'_1,\eps''_1,\eps'_2,\eps''_2)$ on $A$ with values in $(0,1)^4$ and arbitrary continuous functions~$G$ on $A$.

\begin{lemma}\label{Tpp'}
For $1\le p\le 2$, $T^\la_{\overline\eps}$ is bounded from $L^p(A)$ to $L^{p'}(\bR^2)$, uniformly in $\overline\eps$ and $\la$.
\end{lemma}

\begin{proof}
The statement is trivial for $p=1$. 

For $p=2$ we prove the equivalent statement that $(T^\la_{\overline\eps})^*T^\la_{\overline\eps}:L^2(A)\longrightarrow L^2(A)$. We have
$$
(T^\la_{\overline\eps})^*T^\la_{\overline\eps}G(z)=\int_A K^\la_{\overline\eps}(z,w)G(w)\,dw\ ,
$$
where, for $(z,w)\in A^2$,
\bea\label{Keps}
K^\la_{\overline\eps}(z,w)&=\int_{\bR^2} e^{-i(\xi,\eta)\cdot(z-w)}\\
&\qquad \chi_0^4\Big(\frac\xi\la\Big)\chi_0^4\Big(\frac\eta\la\Big)\widehat\chi\big(\eps'_1(z)\xi\big)\widehat\chi\big(\eps'_2(z)\xi\big)\widehat\chi\big(\eps''_1(z)\eta\big)\widehat\chi\big(\eps''_2(z)\eta\big)\\
&\qquad\qquad\qquad\qquad \widehat\chi\big(\eps'_1(w)\xi\big)\widehat\chi\big(\eps'_2(w)\xi\big)\widehat\chi\big(\eps''_1(w)\eta\big)\widehat\chi\big(\eps''_2(w)\eta\big)\,d\xi\,d\eta
\ .
\eea

Let
\bea\label{eps(zwla)}
\eps'(z,w,\la)&=\max\big\{\eps'_1(z),\eps'_2(z),\eps'_1(w),\eps'_2(w),\la\inv\big\}\\
\eps''(z,w,\la)&=\max\big\{\eps''_1(z),\eps''_2(z),\eps''_1(w),\eps''_2(w),\la\inv\big\}\ .
\eea

Using iteratively the property that, given two Schwartz functions $f,g$ on $\bR$, the product $f(at)g(bt)$ can be expressed as $h\big((a\vee b)t\big)$ with each Schwartz norm $\|h\|_{(N)}$ controlled by the same norm of $f$ and $g$,
we can write
\beas
\chi_0^4\Big(\frac\xi\la\Big)\widehat\chi\big(\eps'_1(z)\xi\big)\widehat\chi\big(\eps'_2(z)\xi\big)\widehat\chi\big(\eps'_1(w)\xi\big)\widehat\chi\big(\eps'_2(w)\xi\big)&=\psi'_{z,w,\la}\big(\eps'(z,w,\la)\xi\big)\\
\chi_0^4\Big(\frac\eta\la\Big)\widehat\chi\big(\eps''_1(z)\eta\big)\widehat\chi\big(\eps''_2(z)\eta\big)\widehat\chi\big(\eps''_1(w)\eta\big)\widehat\chi\big(\eps''_2(w)\eta\big)&=\psi''_{z,w,\la}\big(\eps''(z,w,\la)\eta\big)\ ,
\eeas
with $\psi'_{z,w,\la},\psi''_{z,w,\la}\in\cS(\bR)$  uniformly bounded in each Schwartz norm.

Then
\be\label{Ksimpler}
K^\la_{\overline\eps}(z,w)=\frac 1{\eps'(z,w,\la)\eps''(z,w,\la)}\widehat{\psi'}_{z,w,\la}\Big(\frac{z_1-w_1}{\eps'(z,w,\la)}\Big)\widehat{\psi''}_{z,w,\la}\Big(\frac{z_2-w_2}{\eps''(z,w,\la)}\Big)\ ,
\ee
so that, for every $N$, we have the uniform bound
$$
\big|K^\la_{\overline\eps}(z,w)\big|\le C_N \frac1{\eps'(z,w,\la)\eps''(z,w,\la)}\Big(1+\frac{|z_1-w_1|}{\eps'(z,w,\la)}\Big)^{-N}\Big(1+\frac{|z_2-w_2|}{\eps''(z,w,\la)}\Big)^{-N}\ .
$$


We now make a double partition of $A^2$, depending on which of the three parameters $z,w,\la$ determines the value of $\eps'$ and $\eps''$ respectively:
$$
A^2=E'_1\cup E'_2\ ,\qquad A^2=E''_1\cup E''_2\ ,
$$
such that 
$$
\eps'(z,w,\la)=\begin{cases} \eps'_1(z)\text{ or }\eps'_2(z)\text{ or }\la\inv&\text{ on }E'_1\\ \eps'_1(w)\text{ or }\eps'_2(w)&\text{ on }E'_2\ , \end{cases}\qquad \eps''(z,w,\la)=\begin{cases} \eps''_1(z)\text{ or }\eps''_2(z)\text{ or }\la\inv&\text{ on }E''_1\\ \eps''_1(w)\text{ or }\eps''_2(w)&\text{ on }E''_2 \ . \end{cases}
$$

On any intersection $E'_j\cap E''_k=E_{jk}$, each  of $\eps'$ and $\eps''$ depends on only one of the variables $z,w$.
We decompose
\beas
\big|(T^\la_{\overline\eps})^*T^\la_{\overline\eps}G(z)\big|&\le\sum_{j,k=1}^2\int_A {\mathbf 1}_{E_{jk}}(z,w)\big|K^\la_{\overline\eps}(z,w)\big|G(w)\big|\,dw\\
&=\sum_{j,k=1}^2U_{jk}|G|(z)\ ,
\eeas

In the case $j=k=1$ we have
\beas
U_{11}|G|(z)&\le C \int_A \frac1{\tilde\eps'(z)\tilde\eps''(z)}\Big(1+\frac{|z_1-w_1|}{\tilde\eps'(z)}\Big)^{-2}\Big(1+\frac{|z_2-w_2|}{\tilde\eps''(z)}\Big)^{-2}\big|G(w)\big|\,dw\\
&\le CM_s G(z)\ ,
\eeas
where $M_s$ denotes the strong maximal function in $\bR^2$. Hence $U_{11}$ is bounded on $L^2$.

In the case $j=k=2$, it is sufficient to observe that $U_{22}^*$ has the same form as $U_{11}$ to obtain the same conclusion.

Suppose now that $j\ne k$, say $j=1,k=2$, i.e., with $\eps'$ depending on $z$ and $\eps''$ on $w$.
Then, extending $G$ to be 0 on $\bR^2\setminus A$,
\beas
U_{12}|G|(z)&\le C \int_A \frac1{\tilde\eps'(z)\tilde\eps''(w)}\Big(1+\frac{|z_1-w_1|}{\tilde\eps'(z)}\Big)^{-2}\Big(1+\frac{|z_2-w_2|}{\tilde\eps''(w)}\Big)^{-2}\big|G(w)\big|\,dw\\
&= C \int_\bR\frac1{\tilde\eps'(z)}\Big(1+\frac{|z_1-w_1|}{\tilde\eps'(z)}\Big)^{-2}\bigg(\int_\bR \frac1{\tilde\eps''(w)}\Big(1+\frac{|z_2-w_2|}{\tilde\eps''(w)}\Big)^{-2}\big|G(w_1,w_2)\big|\,dw_2\bigg)\,dw_1\\
&= C\int_\bR\frac1{\tilde\eps'(z)}\Big(1+\frac{|z_1-w_1|}{\tilde\eps'(z)}\Big)^{-2} (T|G|)(w_1,z_2)\,dw_1\\
&\le C\,M_1(T|G|)(z_1,z_2)\ ,
\eeas
where 
$M_1f(z_1,z_2)$ denotes the one-dimensional Hardy-Littlewood maximal function of $f(\cdot,z_2)$ evaluated at $z_1$ 
and 
$$
Tf(w_1,z_2)=\int_\bR \frac1{\tilde\eps''(w)}\Big(1+\frac{|z_2-w_2|}{\tilde\eps''(w)}\Big)^{-2}f(w_1,w_2)\,dw_2\ .
$$

In analogy with the previous case, the operator $T^*$,
$$
T^*h(w_1,w_2)=\int_\bR \frac1{\tilde\eps''(w)}\Big(1+\frac{|z_2-w_2|}{\tilde\eps''(w)}\Big)^{-2}h(w_1,z_2)\,dz_2\ ,
$$
 is dominated by 
 $$
\sup_{0<\eps<1}\int_\bR \frac1{\eps}\Big(1+\frac{|z_2-w_2|}{\eps}\Big)^{-2}\big|h(w_1,z_2)\big|\,dz_2= M_2h(w_1,w_2)\ ,
$$
$M_2$ being now the Hardy-Littlewood maximal operator in the second variable.
It follows that $T$, and hence $U_{12}$,  is bounded on $L^2$ and this proves the statement for $p=2$.

The conclusion for $1<p<2$ follows by Riesz-Thorin interpolation.
\end{proof}

We go back to the proof of Theorem \ref{vertical}, recalling that we are assuming $p'=3q$. Observing that $p'/2>2$ and combining together \eqref{square}, \eqref{E2} and Lemma \ref{Tpp'}, we have
$$
\|\cE^\la_{\eps',\eps''} g\|_{L^{p'}(\bR^2)}\le C\big(\|G_+\|_{L^r(A)}+\|G_-\|_{L^r(A)})^\half\ ,
$$
with $G_\pm$ as in \eqref{change} and $r=(p'/2)'=\frac p{2-p}$. To express the right-hand side in terms of the original function $g$, we find that
\beas
\|G_+\|_{L^r(A)}^r&=\int_A \Big|\frac{G_0\big(x_+(z),y_+(z)\big)}{\ph'\big(x_+(z)\big)-\ph'\big(y_+(z)\big)}\Big|^r\,dz\\
&=\int_{U_+} \frac{|G_0(x,y)|^r}{|\ph'(x)-\ph'(y)|^{r-1}}\,dx\,dy\\
&=\int_{U_+} \frac{|g(x)|^r|g(y)|^r}{|\ph'(x)-\ph'(y)|^{r-1}}\kappa(x)^\frac r3\kappa(y)^\frac r3\,dx\,dy\ .
\eeas

Making the change of variables
$$
u=\ph'(x)\ ,\qquad v=\ph'(y)\ ,
$$
and setting $x(u)=(\ph')\inv(u)$, $y(v)=(\ph')\inv(v)$, we obtain that
\beas
\|G_+\|_{L^r(A)}^r&=\int_{\ph'(U_+)} \frac{|g\big(x(u)\big)|^r|g\big(y(v)\big)|^r}{|u-v|^{r-1}}\kappa\big(x(u)\big)^{\frac r3-1}\kappa\big(y(v)\big)^{\frac r3-1}\,du\,dv\ .
\eeas

Notice that $1\le r<2$, so that we can interpret, up to a constant factor, the integral as the pairing $\lan I^{2-r}f,f\ran$, where $I^\alpha$ denotes fractional integration of order $\alpha$ and $f(u)=|g\big(x(u)\big)|^r\kappa\big(x(u)\big)^{\frac r3-1}$. By the Hardy-Littlewood-Sobolev inequality, 
$$
\|G_+\|_{L^r(A)}^r\le C_r\|f\|_{L^s(\ph'(U_+))}^2\ ,
$$
with $s=\frac 2{3-r}$. The same estimate holds for $G_-$, so that, for this value of $s$,
\beas
\|\cE^\la_{\eps',\eps''} g\|_{L^{p'}(\bR^2)}&\le C_p\|f\|_{L^s(\ph'(U))}^\frac1r\\
&=C_p\Big(\int_{\ph'(U)}|g\big(x(u)\big)|^\frac{2r}{3-r}\kappa\big(x(u)\big)^{-\frac 23}\,du\Big)^\frac{3-r}{2r}\\
&=C_p\Big(\int_U|g(x)|^\frac{2r}{3-r}\kappa(x)^{\frac 13}\,du\Big)^\frac{3-r}{2r}\\
&=\|g\|_{L^\frac{2r}{3-r}(I,\mu)}\ .
\eeas

But  $\frac{2r}{3-r}=(p'/3)'=q'$ with $q$ as in the statement of the theorem.
\end{proof}

Consider now the truncated strong maximal function of $\widehat f$,
\be\label{tildeM}
\cM^+f(x)=\sup_{0<\eps',\eps''<1/4}\frac1{4\eps'\eps''}\int_{|s|<\eps',|t|<\eps''}\big|\widehat f\big(x+s,\ph(x)+t\big) \big|\,ds\,dt\ ,\qquad x\in I.
\ee

From Theorem \ref{vertical} we obtain the following inequality for $ \cM^+$ for a more restricted range of $p$.

\begin{corollary}\label{tildeM}
The inequality
\bea
\| \cM^+f\|_{L^q(I,\mu)}\le C_p\|f\|_{L^p(\bR^2)}\ ,\qquad f\in\cS(\bR^2),
\eea
holds for $1\le p<\frac87$ and $p'\ge 3q$.
\end{corollary}

\begin{proof}
As before, we assume $p'=3q$.
Let $h=f*f^*$, where $f^*(x,y)=\overline{f(-x,-y)}$. Then $\widehat h=|\widehat f|^2$, so that $\|h\|_r\le\|f\|_p^2$, with $r=\frac p{2-p}<\frac43$. Then , for $s$ such that $r'=3s$, $\|\cM h\|_s\le C_r\|f\|_p^2$. But, for $\eps',\eps''<\frac14$ and $\chi$ as in \eqref{Mv},
\beas
\frac1{4\eps'\eps''}\int_{|s|<\eps',|t|<\eps''}\big|\widehat f\big(x+s,\ph(x)+t\big) \big|\,ds\,dt&\le \Big(\frac1{4\eps'\eps''}\int_{|s|<\eps',|t|<\eps''}\big|\widehat f\big(x+s,\ph(x)+t\big) \big|^2\,ds\,dt\Big)^\half\\
&=\Big(\frac1{4\eps'\eps''}\int_{|s|<\eps',|t|<\eps''}\widehat h\big(x+s,\ph(x)+t\big) \,ds\,dt\Big)^\half\\
&\le \Big( \int \widehat h\big(x+s,\ph(x)+t\big) \chi_{4\eps'}(s)\chi_{4\eps''}(t)\,ds\,dt\Big)^\half\\
&\le \big(\cM h(x)\big)^\half\ .
\eeas

Hence $\| \cM^+f\|_{L^q(I,\mu)}\le \|\cM h\|_{q/2}^\half$ and it can be easily checked that $q/2=s$.
\end{proof}

\bigskip

\section{Lebesgue points of $\widehat f$ along a curve}\label{sec-differentiation}
\medskip

Adapting standard arguments, cf. [S], we obtain the following  reformulation of Theorem \ref{lebesgue} (ii), where $B_\eps$ denotes the disk of radius $\eps$ centered at $0$.

\begin{corollary}
Let $1\le p<\frac87$ and $S$ be a $C^2$ curve in the plane. Given $f\in L^p(\bR^2)$, for almost every $x\in I$ relative to affine arclength,
$$
\lim_{\eps\to0}\frac1{|B_\eps|}\int_{B_\eps}\big|\widehat f\big(x+x',\ph(x)+y'\big)-\cR f\big(x,\ph(x)\big)\big|\,dx'\,dy'=0\ .
$$
\end{corollary}

\begin{proof} We may restrict ourselves to a subset of $S$ which is the graph of a $C^2$ function $\ph$ on an interval $I$ with $\ph''\ne0$. Let $\mu$ be as in Section \ref{sec-vertical}.

Given $\tau>0$, let $g\in\cS(\bR^2)$ such that $\|f-g\|_p<\tau$. Since $\cR g=\widehat g_{|_S}$,
\beas
F(x)&=\limsup_{\eps\to0}\frac1{|B_\eps|}\int_{B_\eps}\big|\widehat f\big(x+x',\ph(x)+y'\big)-\cR f\big(x,\ph(x)\big)\big|\,dx'\,dy'\\
&\le \limsup_{\eps\to0}\frac1{|B_\eps|}\int_{B_\eps}\big|\widehat {(f-g)}\big(x+x',\ph(x)+y'\big)\big|\,dx'\,dy'+\big|\cR (f-g)\big(x,\ph(x)\big)\big|\\
&\le \cM^+(f-g)(x)+\big|\cR (f-g)\big(x,\ph(x)\big)\big|\ .
\eeas

Hence, if $q=p'/3$, $\|F\|_{L^q(I,\mu)}\le C\tau$ for every $\tau>0$, i.e., $F=0$ $\mu$-a.e.
\end{proof}

\end{document}